\documentclass[12pt,a4paper,oneside,final,reqno]{amsart}

\usepackage{amsthm}
\usepackage{amsmath}
\usepackage{amssymb}
\usepackage{mathtools}
\usepackage{amscd}
\usepackage[utf8]{inputenc}
\usepackage{typearea}
\usepackage{eufrak}
\usepackage{yfonts}
\usepackage{textcomp}
\usepackage{mathrsfs}
\usepackage{hyperref}
\usepackage[draft]{fixme} 
\usepackage{pdfsync}
\usepackage[active]{srcltx}
\usepackage{xypic}
\usepackage{verbatim}
\hypersetup{hidelinks}

\usepackage{todonotes}
\usepackage[capitalise]{cleveref}
\usepackage{tikz-cd}
\usepackage{tensor}

\usepackage[backend=bibtex, style=alphabetic, maxnames=50]{biblatex}

\newcommand{\C}{\mathbb{C}}

\newcommand{\N}{\mathbb{N}}
\newcommand{\R}{\mathbb{R}}

\def\acts{\curvearrowright}

\theoremstyle{plain}
\newtheorem{theorem}{Theorem}[section]
\newtheorem{corollary}[theorem]{Corollary}
\newtheorem{lemma}[theorem]{Lemma}
\newtheorem{proposition}[theorem]{Proposition}
\newtheorem{remark}[theorem]{Remark}

\newtheorem{introtheorem}{Theorem}

\theoremstyle{definition}
\newtheorem{definition}[theorem]{Definition}

\title[CQMS from length functions on hyperbolic groups]{Compact Quantum Metric Spaces from Weakly Geodesic Length Functions on Hyperbolic Groups}

 \author[1]{Are Austad}
 \address{Are Austad, Department of Mathematics, University of Oslo, P.O.Box 1053 Blindern, 0316 Oslo, Norway}
 \email{areaus@math.uio.no}

\begin{document}

\maketitle

\begin{abstract}
    We show that length functions on hyperbolic groups whose induced metrics are hyperbolic and weakly geodesic naturally give rise to compact quantum metric spaces, thus generalizing the result of Ozawa and Rieffel. As a result we obtain several new examples of compact quantum metric spaces, in particular arising from Green metrics associated with random walks on hyperbolic groups, and from proper cocompact isometric  actions on proper geodesic hyperbolic spaces admitting a point with trivial stabilizer. 
\end{abstract}

\section{Introduction}\label{sec:introduction}
The theory of compact quantum metric spaces, initiated by Rieffel \cite{RieffelMetricsOnStateSpaces1999, RieffelCQMS04, Rieffel2004qGH, Rieffel2004MatrixAlgs, RieffelMartricialBridges2016}, provides a noncommutative analogue of the theory of compact metric spaces. One of the principal sources of examples comes from spectral triples in the sense of Connes \cite{Connes1989, ConnesNCGBook}. Indeed, if $(\mathcal{A},H,D)$ is a unital spectral triple, the commutator with the Dirac operator defines a seminorm
\begin{align*}
    L_D(a):=\Vert [D,a]\Vert ,\quad a\in\mathcal A,
\end{align*}
and one may ask whether the corresponding Monge--Kantorovič metric metrizes the weak$^*$ topology on the state space. When this is the case, the pair $(\mathcal{A},L_D)$ is said to be a compact quantum metric space.

Determining whether a given spectral triple yields a compact quantum metric space is typically a non-trivial problem. It has been studied in the contexts of $C^*$-algebras arising from groups \cite{Rieffel02, OzawaRieffel2005, ChristRieffel2017, ChristensenIvanRD, FarsiLandryLarsenPacker}, compact quantum groups \cite{BhowmickVoigtZacharias2015, AguilarKaad2018, KaadKyedSU2, AustadKyed2026}, crossed products \cite{BellissardMarcolliReihani2010, HawkinsSkalskiWhiteZacharias2013, KlisseCQMS, AustadKaadKyed2025, GerontogiannisMesland25}, and groupoids \cite{AustadGpdCQMS}. Of particular relevance here are compact quantum metric spaces arising from length functions on discrete groups. Suppose $G$ is a countable discrete group equipped with a proper length function $\ell \colon G \to [0,\infty)$. The length function gives rise to a self-adjoint unbounded operator $D$ which on the canonical basis $(\delta_g)_{g\in G}$ is given by
\begin{align*}
    D \delta_g = \ell(g) \delta_g, \quad g\in G.
\end{align*}
This operator interacts nicely with the left regular representation $\lambda$, in that $[D, \lambda(f)]$ extends to an element of $B(\ell^2(G))$ for every $f \in C_c(G)$. We denote the extension by $[D , \lambda(f)]$ as well, and we obtain a seminorm on $C_c(G)$ given by
\begin{align}\label{eq:def-L-intro}
    L(f) := \Vert [D, \lambda(f)]\Vert .
\end{align}
A natural question is therefore:
\begin{center}
    \emph{Under which assumptions on $G$ and $\ell$ is the pair $(C_c(G), L)$ a compact quantum metric space?}
\end{center}
There are several positive results in the literature. Ozawa and Rieffel showed in \cite{OzawaRieffel2005} that one gets a compact quantum metric space this way whenever $G$ is word hyperbolic and $\ell$ is the word length function. Christ and Rieffel \cite{ChristRieffel2017} moreover found that such a quantum structure is always present whenever the pair $(G,\ell)$ has bounded doubling. In particular, this includes the case of $G$ a finitely generated virtually nilpotent group and $\ell$ the word length function by Gromov's theorem on polynomial growth. Interestingly, both of these classes belong to the class of metric groups which have rapid decay. \cite{ChristensenIvanRD} showed that for a pair $(G,\ell)$ with rapid decay, there exists an integer $k = k_{(G,\ell)}$ such that if $L$ is defined analogously to \eqref{eq:def-L-intro} but instead using $k$ iterated commutators, then $(C_c(G), L)$ is indeed a compact quantum metric space. It remains an open problem whether one may always choose $k=1$ for finitely generated pairs $(G,\ell)$ with rapid decay. Recently, Klisse exhibited the first example of a finitely generated group with a word length function for which $(C_c(G), L)$ is not a compact quantum metric space \cite{Klisse26}. This example lies outside the rapid decay setting.  The corresponding question for groups with rapid decay remains open.

Also recently, the author has tied the question to an approximation criterion in terms of existence of finitely supported Fourier multipliers, allowing for a new approach to the problem for arbitrary proper length functions \cite{AustadGpdCQMS}. 
We record the relevant result from \cite{AustadGpdCQMS} as
Proposition \ref{prop:CQMS-equiv-weak-amenable-gps}. 

The aim of this article is to revisit Ozawa and Rieffel's result and loosen the previous restriction to word length functions in order to accommodate other interesting length functions on hyperbolic groups. Such length functions arise, for example, from random walks on the group, or from proper cocompact isometric actions on hyperbolic spaces with a point having trivial stabilizer. Three key properties of word length functions on hyperbolic groups are used in \cite{OzawaRieffel2005}: They are proper, hyperbolic, and geodesic. The main contribution of this paper is to show that geodesicity can be relaxed to weak geodesicity in the sense of Kasparov and Skandalis \cite{KasparovSkandalis2003}. In particular the length functions need no longer be integer-valued. 
Note that any length function $\ell \colon G \to [0,\infty)$ induces a left-invariant metric $d_\ell$ on $G$ given by $d_\ell(g, h) = \ell(g^{-1}h)$ for $g,h\in G$. 
Our main result is the following. 

\begin{introtheorem}\label{introthm:main-theorem}
    Let $G$ be a finitely generated hyperbolic group, and suppose $\ell\colon G \to [0,\infty)$ is a proper length function for which $(G,d_\ell)$ is a hyperbolic and weakly geodesic metric space. Then the seminorm
    \begin{align*}
        L(f) = \Vert [D, \lambda(f)]\Vert , \quad f\in C_c(G),
    \end{align*}
    makes $(C_c(G), L)$ a compact quantum metric space. 
\end{introtheorem}

Generalizing from word length functions to the setting of Theorem \ref{introthm:main-theorem} comes with additional difficulties. Key to the proof of the result concerning word hyperbolic groups in \cite{OzawaRieffel2005} is what they call a Haagerup-type estimate, the proof of which in particular uses the geodesic structure of the Cayley graph. The first difficulty to overcome is therefore adapting this result to the weakly geodesic setting, see Lemma \ref{lemma:annular-norm-control}. Explicitly, denoting by $E_k := \{x \in G \mid k \leq \ell(x) < k+1\}$ and letting $P_k$ be the associated orthogonal projection on $\ell^2(G)$, we show that there exists some constant $C >0$, depending only on 
$G$ and $\ell$, 
such that
\begin{align*}
    \Vert P_m \lambda(f) P_n \Vert \leq C \Vert f \Vert_2
\end{align*}
for all $m,n,k$ and all $f$ supported on $E_k$. 
The proof replaces exact points along a geodesic by approximate intermediate points determined by weak geodesicity, and then leverages hyperbolicity to keep uniform control over the multiplicities of elements appearing in factorizations of group elements. 

Another difficulty arises from the length functions merely being real-valued. In the integer-valued setting of \cite{OzawaRieffel2005}, the authors provide a continuous off-diagonal cutoff realized through a discrete Fourier series argument. This argument hinges crucially on the length function being integer-valued, which allows them to view it as an action of the circle group. Neither the argument nor the result is available when using non-integer-valued length functions, and we instead use a continuous Fourier cutoff associated with the one-parameter group
\begin{align*}
    \alpha_t (T) := e^{2\pi i t D} T e^{-2\pi i t D},\quad T \in B(\ell^2(G)). 
\end{align*}
By choosing a suitable compactly supported function $h \in C_c^\infty(\R)$ and $R \geq 1$, the resulting operator $Q_{h,R}(\lambda(f))$ for $f \in C_c(G)$ will satisfy
\begin{align*}
    P_m Q_{h,R}(\lambda(f)) P_n = 0 \quad \text{for $|m-n| >  2R +2 $}
\end{align*}
and 
\begin{align*}
    \Vert \lambda(f) - Q_{h,R}(\lambda(f)) \Vert \leq \frac{C_h}{R} \Vert [D , \lambda(f)] \Vert,
\end{align*}
for some $C_h \geq 0$ depending only on $h$. By combining these cutoffs with the generalized Haagerup-type condition we obtain a norm estimate which puts us in a position to apply the Fourier multiplier criterion from \cite{AustadGpdCQMS}. We then invoke weak amenability of hyperbolic groups, due to Ozawa \cite{OzawaHyperbolicWeakAmenable2008}, to supply the required family of finitely supported multipliers. 

Theorem \ref{introthm:main-theorem} produces compact quantum metric spaces from several other geometrically natural length functions on hyperbolic groups, of which we treat two in particular. Note that even though both  length functions are quasi-isometric to word length functions, the fact that they give rise to compact quantum metric structures does not follow from quasi-isometry. 
The first family of examples comes from random walks. Let $G$ be a non-elementary hyperbolic group and $\mu$ a symmetric finitely supported probability measure whose support generates $G$. The Green metric is then defined as
\begin{align*}
    d_\mu (x,y) = - \log F_\mu (x,y),
\end{align*} 
where $F_\mu(x,y)$ is the probability of ever hitting $y$ when starting at $x$.
Then $d_\mu $ is a left-invariant hyperbolic metric which is quasi-isometric to a word metric. In particular it follows that it is proper and weakly geodesic. It follows that the length function
\begin{align*}
    \ell_\mu (g ) = d_\mu (e,g)
\end{align*}
satisfies the conditions of Theorem \ref{introthm:main-theorem}, and we therefore obtain compact quantum metric spaces from symmetric finitely supported random walks on hyperbolic groups. 

The second family of examples arises from geometric actions. Suppose that a hyperbolic group $G$ acts properly, cocompactly and isometrically on a proper geodesic hyperbolic metric space $(X,d)$, and let $x_0 \in X$ be a point with trivial stabilizer. We then obtain a proper length function on $G$ through
\begin{align*}
    \ell_{x_0}(g) := d (x_0, gx_0)
\end{align*}
and the induced metric $d_{\ell_{x_0}}$ identifies $G$ isometrically with the orbit $Gx_0 \subseteq X$. Thus $(G,d_{\ell_{x_0}})$ is hyperbolic. Moreover, the \v{S}varc--Milnor lemma yields that $d_{\ell_{x_0}}$ is quasi-isometric to a word metric, and so $d_{\ell_{x_0}}$ is weakly geodesic. Applying Theorem \ref{introthm:main-theorem} we therefore get compact quantum metric space structures on $C_c(G)$ from these actions. 

These examples illustrate that there are natural choices of length functions other than the word length which encode geometric or probabilistic information for a hyperbolic group. 

The article is organized as follows. In Section \ref{sec:prelims} we collect necessary background material concerning length functions on groups, hyperbolicity and weak geodesicity, as well as basic results about compact quantum metric spaces. Most of Section \ref{sec:results} is devoted to the proof of Theorem \ref{introthm:main-theorem}. Along the way, we prove the generalized Haagerup condition Lemma \ref{lemma:annular-norm-control}, and the continuous Fourier cutoff needed to control the off-diagonal part of the regular representation in Lemma \ref{lemma:smooth-offdiagonal-cutoff}. Lastly, we apply Theorem \ref{introthm:main-theorem} to obtain compact quantum metric spaces from Green metrics associated with symmetric finitely supported random walks, and from length functions arising from proper cocompact isometric actions on proper geodesic hyperbolic spaces. 

\subsubsection*{Acknowledgments} The author was 
funded by the Research Council of Norway [project 359602].

\section{Preliminaries}\label{sec:prelims}

\subsection{Length functions and hyperbolicity}
Let $G$ be a countable discrete group, and denote by $e$ the unit of $G$. By a length function on $G$ we mean a map $\ell \colon G \to [0,\infty)$ such that
\begin{enumerate}
    \item $\ell(x) = 0$ if and only if $x = e$, 
    \item $\ell(x) = \ell(x^{-1})$ for all $x \in G$, 
    \item $\ell(xy) \leq \ell(x) + \ell(y)$ for all $x,y \in G$.
\end{enumerate}
We furthermore say $\ell$ is proper if $B_\ell(r) := \ell^{-1}([0,r])$ is a finite set for all $ r \geq 0$. Two length functions $\ell, \ell' \colon G \to [0,\infty)$ are said to be quasi-isometric if there exist $A \geq 1$ and $B \geq 0$ such that
\begin{align*}
    A^{-1} \ell(x) -B \leq \ell' (x) \leq A \ell(x) + B
\end{align*}
for all $x \in G$. 

We proceed to recall the definition of a hyperbolic metric space.
\begin{definition}\label{def:hyperbolic-metric-space}
    A metric space $(X,d)$ is said to be \emph{hyperbolic} (or $\delta$-hyperbolic) if there is $\delta \geq 0$ such that for any four points $x,y,z,w \in X$ we have
    \begin{align*}
        d(x,y) + d(z,w) \leq \max \{ d(x,z) + d(y,w), d(x,w) + d(y,z) \} + \delta .
    \end{align*}
\end{definition}
For a length function $\ell \colon G \to [0,\infty)$ we denote by $d_\ell$ the induced left-invariant metric on $G$
\begin{align*}
    d_\ell(x,y) = \ell(x^{-1}y).
\end{align*}
We shall be interested in length functions $\ell$ for which $(G, d_\ell)$ is $\delta$-hyperbolic for some $\delta \geq 0$. If $G$ is a finitely generated group with finite symmetric generating set $S$, and $\ell = \ell_S$ is the corresponding word length function, we say that $G$ is (word) hyperbolic if $(G, d_\ell)$ is $\delta$-hyperbolic for some $\delta \geq 0$.

For the purposes of our generalization of the argument from \cite[Proposition 4.3]{OzawaRieffel2005}, we will also need a weaker version of geodesicity. In \cite[Definition 2.1]{KasparovSkandalis2003}, a metric space $(X,d)$ is said to be $c$-weakly geodesic if for all $x,y \in X$ and all $r \in [0,d(x,y)]$ there is $z \in X$ such that
\begin{align*}
    d(z,x) \leq r +c, \quad d(z,y) \leq d(x,y) - r + c.
\end{align*}
We will consider length functions whose induced metrics have this property. 

    Using left-invariance of the metric $d_\ell$, it is not difficult to show that $\ell$ induces a weakly $c$-geodesic metric if and only if for all $x \in G$ and all $0 \leq r \leq \ell(x)$ there is $\overline{x} \in G$ such that
\begin{align}\label{eq:c-weak-geodesic-conditions}
    \vert \ell(\overline{x}) - r \vert \leq c \quad \text{and} \quad \vert \ell(\overline{x}^{-1}x) - (\ell(x) - r) \vert \leq c.
\end{align}
 
Beyond word length functions, we will see other examples of length functions inducing weakly geodesic metrics in Section \ref{subsec:Examples}. 

\begin{definition}\label{def:weakly-geodesic-length-function}
    We say that a length function $\ell \colon G \to [0,\infty)$ \emph{induces a weakly $c$-geodesic metric} if $(G,d_\ell)$ is weakly $c$-geodesic. We say $\ell$ induces a weakly geodesic metric if it induces a weakly $c$-geodesic metric for some $c \geq 0$. 
\end{definition}

We make the following short observation about groups admitting proper length functions inducing weakly geodesic metrics. 

\begin{lemma}\label{lemma:proper-weak-geodesic-is-fin-gen}
    Let $G$ be a countable discrete group, and let $\ell \colon G \to [0,\infty)$ be a proper length function inducing a weakly $c$-geodesic metric. Then $G$ is finitely generated and $\ell$ is quasi-isometric to a word length function on $G$. More precisely, if
    \begin{align*}
        S:=B_\ell(1+2c)\setminus\{e\},
    \end{align*}
    then $S$ is a finite symmetric generating set for $G$, and
    \begin{align*}
        \frac{1}{1+2c}\ell(g)\leq \ell_S(g)\leq \ell(g)+1
    \end{align*}
    for all $g \in G$.
\end{lemma}
\begin{proof}
    Since $\ell$ is a proper length function, $S$ is finite and symmetric. We first show $S$ generates $G$. 
    Now, if $\ell(g) > 1+c$, then apply weak geodesicity with $r = 1+c$ to find $h$ such that
    \begin{align*}
        \ell(h) \leq 2c+1 , \quad \ell(h^{-1}g) \leq \ell(g) -1.
    \end{align*}
    Iterating this process, we see that any $g$ can be written as a product of elements in $S$, that is, $S$ generates $G$. 
    
    Moreover, the argument shows that $\ell_S(g) \leq \ell(g) +1$. Conversely, as $\ell(s) \leq 1+2c$ for all $s \in S$, we have the trivial estimate $\ell(g) \leq (1+2c) \ell_S(g)$. 
    The result follows. 
\end{proof}

We will also need a result from \cite[pg. 1858]{CantrellTanaka2025}, citing \cite[Proposition 5.6]{BonkSchramm2000}, for which we need to relate the notion of weak geodesicity to rough geodesicity. 
\cite{CantrellTanaka2025} says a metric is $C$-roughly geodesic if for any pair $x,y \in G$ there is an interval $I$ and a map $\gamma \colon I \to G$ with $x$ and $y$ as endpoints, as well as
\begin{align*}
    |s-t| - C \leq d(\gamma(s), \gamma(t)) \leq |s-t|+C \quad \forall s,t\in I.
\end{align*}
This implies that the conditions of \eqref{eq:c-weak-geodesic-conditions} are satisfied: Let $x \in G$ be given, and let $0 \leq r \leq \ell(x)$. Let $\gamma \colon [a,b] \to G$ be a $C$-rough geodesic with endpoints given by $\gamma(a)= e$ and $\gamma(b) = x$. Note that
\begin{align*}
    |(b-a) - \ell(x)| \leq C
\end{align*}
by the above condition. Given $0 \leq r \leq \ell(x)$, pick
\begin{align*}
\overline{x} =
    \begin{cases}
        \gamma(a+r) & \text{if $a+r \in [a,b]$} \\
        \gamma(b)=x & \text{otherwise}. 
    \end{cases}
\end{align*}
We then easily check that $\vert \ell(\overline{x}) - r \vert \leq C$ in both cases. Moreover, using $|(b-a) - \ell(x)| \leq C$ from above, we may establish that
\begin{align*}
    -2C \leq \ell(\overline{x}^{-1}x) - (\ell(x) - r) \leq 2C.
\end{align*}
Hence every $C$-roughly geodesic left-invariant metric is weakly $2C$-geodesic in the sense of Definition \ref{def:weakly-geodesic-length-function}.  
We may now state the following observation noted in \cite[pg. 1858]{CantrellTanaka2025}. 
\begin{lemma}\label{lemma:rough-geodesic-lfs-on-hyperbolic-groups}
     Every hyperbolic left-invariant metric on a hyperbolic group which is quasi-isometric to a word metric is itself roughly geodesic. In particular this is true for $d_\ell$ for $\ell \colon G \to [0,\infty)$ a length function which is quasi-isometric to a word length function $\ell_S$ for some finite symmetric generating set $S$ for $G$. Consequently, if $\ell$ is a length function which is quasi-isometric to a word length function and for which $d_\ell$ is hyperbolic, then $d_\ell$ is also weakly geodesic. 
\end{lemma}

\subsection{Compact quantum metric spaces and group $C^*$-algebras}\label{subsec:prelims-CQMS}
As we will only be interested in quantum metrics on $C_c(G)$, we treat compact quantum metric spaces in the setting of dense $*$-subalgebras of $C^*$-algebras. 
Note that if $\mathcal{A}$ is a unital and dense $*$-subalgebra of a unital $C^*$-algebra $A$, then we may identify $S(\mathcal{A})$ and $S(A)$ through the restriction map, where $S(\mathcal{A})$ is the state space of the $*$-algebra $\mathcal{A}$. Positivity of elements is interpreted in the ambient $C^*$-algebra. 

The following terminology is from \cite{RieffelMartricialBridges2016}.
\begin{definition}
    Let $\mathcal{A}$ be a unital and dense $*$-subalgebra of a unital $C^*$-algebra $A$. A seminorm $L \colon \mathcal{A} \to [0,\infty)$ is a \emph{slip-norm} if it satisfies
    \begin{enumerate}
        \item $L(a^*) = L(a)$ for all $a \in \mathcal{A}$,
        \item $\C \cdot 1_{\mathcal{A}} \subseteq \ker L := \{a \in \mathcal{A} \mid L(a) = 0 \}$. 
    \end{enumerate}
\end{definition}
Given a slip-norm $L$ on $\mathcal{A}$, we consider the associated \emph{Monge--Kantorovi\v{c}} metric $\operatorname{mk}_L$ on the state space $S(\mathcal{A})$, given by
\begin{align*}
    \operatorname{mk}_L (\phi, \psi ) := \sup \{ \vert \phi(a) - \psi(a) \vert \mid L(a) \leq 1 \}.
\end{align*}
Note that despite calling $\operatorname{mk}_L$ a metric, it is a priori just an extended metric on $S(\mathcal{A})$. 

\begin{definition}\label{def:CQMS}
    Suppose $\mathcal{A}$ is a dense and unital $*$-subalgebra of a unital $C^*$-algebra $A$, and suppose $L \colon \mathcal{A} \to [0, \infty)$ is a slip-norm. If $\operatorname{mk}_L$ induces the weak$^*$ topology on $S(\mathcal{A})$, we say that $(\mathcal{A}, L)$ is a \emph{compact quantum metric space}. 
\end{definition}

Now, fix a countable discrete group $G$. To such a group we may associate a $*$-algebra $C_c(G)$ consisting of the finitely supported functions on $G$, equipped with convolution and involution
\begin{align*}
    f_1*f_2 (g) = \sum_{h\in G} f_1(h) f_2 (h^{-1}g), \quad f^*(g) = \overline{f(g^{-1})}, \quad 
    \text{for $f, f_1, f_2 \in C_c(G)$}. 
\end{align*}
The left regular representation $\lambda \colon C_c(G) \to B(\ell^2(G))$ is given by
\begin{align*}
    \lambda(f)(\xi) = f * \xi 
\end{align*}
for $\xi \in \ell^2(G)$ and $f \in C_c (G)$, and the resulting $C^*$-completion of $C_c(G)$ is the reduced group $C^*$-algebra, denoted by $C^*_r(G)$. 

Given a proper length function $\ell \colon G \to [0,\infty)$ we obtain an unbounded self-adjoint operator on $\ell^2(G)$, given by the self-adjoint closure of the operator $D \colon C_c(G) \to \ell^2(G)$ which acts as
\begin{align*}
    D \delta_g = \ell(g) \delta_g,
\end{align*}
where $\{\delta_g\}_{g\in G}$ is the canonical basis for $\ell^2(G)$. We denote the self-adjoint extension by $D$ as well. The operator $D$ interacts nicely with the left regular representation in that for every $f \in C_c(G)$, the operator $[D, \lambda(f)]$ extends to an element of $B(\ell^2(G))$. We may therefore define the seminorm
\begin{align}\label{eq:def-L-ell}
    L \colon C_c (G) \to [0,\infty), \quad L(f) = \Vert [D, \lambda(f)]\Vert.
\end{align}
We wish to understand when $(C_c(G), L)$ is a compact quantum metric space, under the assumption that $G$ is a hyperbolic group and $\ell\colon G \to [0,\infty)$ is a proper length function. It is not difficult to see that $L$ is a slip-norm, and as such it remains to verify that the Monge-Kantorovi\v{c} metric induces the weak$^*$ topology on $S(C_c(G))$. We will not do so directly, but instead apply a result from \cite{AustadGpdCQMS} inspired by \cite[Proposition 1.3]{OzawaRieffel2005}, see Proposition \ref{prop:CQMS-equiv-weak-amenable-gps}.

Any $\phi \in C_c(G)$ defines a Fourier multiplier 
\begin{align*}
    M_\phi \colon C^*_r(G) \to C^*_r(G), \quad M_\phi(f)(g) = \phi(g) f(g),
\end{align*}
which is completely bounded. We denote the completely bounded norm by $\Vert M_\phi \Vert_{\rm cb}$. A countable discrete group is said to be \emph{weakly amenable} if it admits a sequence $(\phi_n)_n$ of finitely supported functions converging pointwise to $1$, and for which $\phi_n(e) = 1$ for all $n$ and $\sup_{n}\Vert M_{\phi_n}\Vert_{\rm cb} < \infty$. 
By \cite[Corollary 3.16]{AustadGpdCQMS} we have the following
\begin{proposition}\label{prop:CQMS-equiv-weak-amenable-gps}
    Suppose $G$ is a weakly amenable group, and let $(\phi_n)_n$ be a sequence witnessing the weak amenability. Let $\ell$ be any proper length function, and denote by $E$ the set
    \begin{align*}
        E := \{f \in C_c(G) \mid f(e)=0 \text{ and } L(f) \leq 1 \}.
    \end{align*}
    Then $(C_c(G), L)$ is a compact quantum metric space if and only if for every $\varepsilon > 0$ there is $N = N_\varepsilon$ such that
    \begin{align*}
        \sup_{f \in E} \Vert f - M_{\phi_n}(f) \Vert < \varepsilon
    \end{align*}
    for all $n \geq N$. 
\end{proposition}
Note that since $\varphi_n(e)=1$, the multiplier $M_{\varphi_n}$ fixes scalar multiples of $\lambda_e$. Moreover, $L(f-f(e)\lambda_e)=L(f)$. Hence the criterion in \cite[Corollary 3.16]{AustadGpdCQMS} is unchanged if the supremum is restricted to functions satisfying $f(e)=0$.

We remark that hyperbolic groups are weakly amenable by \cite{OzawaHyperbolicWeakAmenable2008}.

\section{Results}\label{sec:results}

\subsection{A generalized Haagerup-type condition}
The following lemma will be key to proving Theorem \ref{thm:main-theorem}. The proof follows the philosophy of \cite[Proposition 4.3]{OzawaRieffel2005}, but is complicated by the lack of true geodesicity.

\begin{lemma}\label{lemma:annular-norm-control}
   Let $G$ be a 
   countable discrete
   group, and suppose $\ell \colon G \to [0,\infty)$ is a proper length function such that $(G,d_\ell)$ is a $\delta$-hyperbolic and weakly $c$-geodesic metric space. 
   Denote by $E_m := \{x \in G \mid m \leq \ell(x) < m+1 \}$, 
   and let 
   $P_m \in B(\ell^2(G))$ be the orthogonal projection onto $\ell^2(E_m)$. 
   Then there exists $C \geq 0$, depending only on $G$ and $\ell$, such that for all $m,n,k \in \N_0$ and any $f \in C_c(G)$ supported on $E_k$, we have
   \begin{align*}
       \Vert P_m \lambda(f) P_n \Vert \leq C \Vert f \Vert_2. 
   \end{align*}
   Explicitly, we may take any $C \geq |B_\ell(2c + \frac{7}{2}+2\delta)\vert$. 
\end{lemma}
\begin{proof}
    We need to show that
    \begin{align*}
        \Vert P_m(f * \xi) \Vert \leq C \Vert f\Vert_2 \Vert \xi \Vert_2
    \end{align*}
    for any $\xi \in \ell^2(G)$  supported on $E_n$. Fix such a $\xi$, and set $X_{m,n,k} := E_m \cap E_kE_n$. If $x \in E_m \setminus X_{m,n,k}$, then $f*\xi(x) = 0$, and so it suffices to consider $x \in X_{m,n,k}$. For such $x$, let $y \in E_k$ and $z \in E_n$ be arbitrary elements such that $yz = x$. 

    Now set
    \begin{align*}
        a_x = \frac{k+n+1 - \ell(x)}{2} , \quad b_x = \frac{k-n +\ell(x)}{2} .
    \end{align*}
     Since $\ell(x) \geq \ell(z) - \ell(y) > n - (k+1)$ for $y \in E_k$ and $z \in E_n$, we see that $b_x > -\frac{1}{2}$. 
    Using $\ell(x) \geq \ell(y) - \ell(z)$ we also establish that $\ell(x) - b_x > -\frac{1}{2}$. 
    
    Now define
    \begin{align*}
        r_x = \begin{cases}
            0 & b_x < 0 \\
            b_x & 0\leq b_x \leq \ell(x) \\
            \ell(x) & b_x > \ell(x)
        \end{cases}
    \end{align*}
    which immediately implies $|r_x - b_x| < \frac{1}{2}$ by the above observations. 
    For each $x \in X_{m,n,k}$, we use weak geodesicity to fix an element $\overline{x} \in G$ satisfying 
    \begin{align*}
        |\ell(\overline{x}) - r_x | \leq c, \quad | \ell(\overline{x}^{-1}x) - (\ell(x) - r_x)| \leq c,
    \end{align*}
    and set $\tilde{x} := \overline{x}^{-1}x$.

    Applying the hyperbolicity criterion from Definition \ref{def:hyperbolic-metric-space} to the points $e, x, \overline{x}, y$, we obtain
    \begin{align*}
        \ell(x) + \ell(y^{-1}\overline{x}) \leq \max \{ \ell(\overline{x}) + \ell(z) , \ell(y) + \ell(\tilde{x}) \} + \delta, 
    \end{align*}
    where we have used $z = y^{-1}x$ and $\tilde{x} = \overline{x}^{-1}x$. Now,
    \begin{align*}
        \ell(\overline{x}) + \ell(z) < r_x + c  + n +1 \leq b_x + \frac{1}{2} + c + n+1 = a_x + \ell(x) + c + 1.
    \end{align*}
    Moreover,
    \begin{align*}
        \ell(y) + \ell(\overline{x}^{-1}x) < k+1 + \ell(x) - r_x +c \leq k+1 +\ell(x) -b_x + \frac{1}{2} +c  = a_x + \ell(x) +c + 1.
    \end{align*}
    This now implies
    \begin{align*}
        \ell(y^{-1}\overline{x}) \leq - \ell(x) + a_x + \ell(x) + c + 1 + \delta = a_x + c +1+\delta
    \end{align*}
    Set $R_x := a_x + c+ 1 +\delta$. 
    Since the preceding estimate holds for every factorization $x = yz$ with $y \in E_k$ and $z \in E_n$, the Cauchy--Schwarz inequality yields
    \begin{align*}
        \Vert P_m(f*\xi)\Vert_2^2
\leq
\sum_{x\in X_{m,n,k}}
\sum_{\ell(u)\leq R_x}
\sum_{\ell(v)\leq R_x}
|f(\bar xu)|^2|\xi(v\widetilde x)|^2.
    \end{align*}

    For each term in the triple sum, define $y = \bar{x}u$ and $z = v\tilde{x}$. We only have nonzero contributions for $y \in E_k$ and $z \in E_n$. 
    We can then write
    \begin{align}\label{eq:multiplicity-convolution-bound-Nyz}
        \Vert P_m (f*\xi) \Vert^2 \leq  \sum_{y \in E_k} \sum_{z \in E_n} N(y,z) \vert f(y) \vert^2 \cdot \vert \xi(z) \vert^2
    \end{align}
    for
    \begin{align*}
        N(y,z) := |\{(x,u,v) \mid x \in X_{m,n,k}, y = \bar{x}u, z= v\tilde{x}\text{ and } \ell(u),\ell(v) \leq R_x \}|.
    \end{align*}
    It suffices to show that $N(y,z)$ is uniformly bounded for all $y$ and $z$. To this end, we define
    \begin{align*}
        \overline{F}_y :&= \{ u \in G \mid  \text{  there is $x \in X_{m,n,k}$ such that $y = \overline{x}u$} \text{ and } \ell(u) \leq R_x \} \\
        \tilde{F}_z :&=  \{ v \in G \mid  \text{  there is $x \in X_{m,n,k}$ such that $z = v\tilde{x}$ and $\ell(v) \leq R_x$} \}.
    \end{align*}
    For fixed $y,z,u,v$, the element $x$ is uniquely determined by $x=\bar x\widetilde x
=yu^{-1}v^{-1}z$.
Hence the map
$(x,u,v)\mapsto (u,v)$ 
is injective on the set counted by $N(y,z)$, and therefore
\begin{align*}
    N(y,z)\leq |F_y|\,|\widetilde F_z|.
\end{align*}  
    We therefore aim to show that $\vert \overline{F}_y \vert $ and $\vert \tilde{F}_z \vert$ are bounded independently of $y$ and $z$.

    Let $u,w \in \overline{F}_y$, and choose $x,x' \in X_{m,n,k}$ such that
    \begin{align*}
        y = \overline{x}u = \overline{x'}w
    \end{align*}
    with $\ell(u) \leq R_x$ and $\ell(w) \leq R_{x'}$. Set $s:= \overline{x}$ and $t:= \overline{x'}$, so that $y = su = tw$. Since $x,x' \in E_m$, we have
    $\vert \ell(x) - \ell(x') \vert < 1$ and therefore $\vert b_x - b_{x'} \vert \leq \frac{1}{2}$. Moreover, because $x,x' \in X_{m,n,k}$, we have
    \begin{align*}
        \vert r_x - b_x \vert < \frac{1}{2} \quad \text{and} \quad \vert r_{x'} - b_{x'} \vert < \frac{1}{2}.
    \end{align*}
    It follows that $\vert r_x - r_{x'}\vert < \frac{3}{2}$ by the triangle inequality. 
    
    By construction we have
    \begin{align*}
        \ell(s) &\leq r_{x} +c, \quad \ell(u) \leq a_x + 1 +c +\delta \\
        \ell(t) &\leq r_{x'} + c, \quad \ell(w) \leq a_{x'} +1+c+\delta.
    \end{align*}
    Now, $a_x = k - b_x + \frac{1}{2}$ and $|b_x - r_x | < \frac{1}{2}$, and similarly for $x'$. Thus
    \begin{align*}
        a_x \leq k+1 - r_x , \quad a_{x'} \leq k+1 - r_{x'},
    \end{align*}
    and so 
    \begin{align*}
        \ell(u) &\leq k-r_x +2+c+\delta \\ 
        \ell(w) &\leq k-r_{x'} +2 +c +\delta.
    \end{align*}
    In particular,
    \begin{align*}
        \ell(s) + \ell(w) &\leq r_x + c + k - r_{x'} +2 +c +\delta \leq k + 2c + \frac{7}{2} + \delta \\
        \ell(t) + \ell(u) &\leq r_{x'} + c + k - r_{x} +2+c+\delta \leq k + 2c  + \frac{7}{2} + \delta
    \end{align*}
    where we have used $|r_x - r_{x'}| < \frac{3}{2}$. 
    We then apply the hyperbolicity condition to $e,y,s,t$ to obtain
    \begin{align*}
        \ell(y) + \ell(s^{-1}t) \leq \max \{ \ell(s) + \ell(w), \ell(t) + \ell(u) \} +\delta ,
    \end{align*}
    and since $\ell(y) \geq k$, we deduce
    \begin{align*}
        \ell(u w^{-1}) = \ell(s^{-1}t) \leq 2c  + \frac{7}{2} + 2\delta.
    \end{align*}

    If $\overline{F}_y = \emptyset$, the bound is immediate. Otherwise,
    fix $w \in \overline{F}_y$. The map $\overline{F}_y \to G$ , $u \mapsto uw^{-1}$
    is injective and takes values in $B_\ell(2c +\frac{7}{2} + 2\delta )$. Thus
    \begin{align*}
        \vert \overline{F}_y \vert \leq | B_\ell(2c +\frac{7}{2} + 2\delta ) \vert.
    \end{align*}
    Since $\ell$ is proper, this is a finite value, independent of $y$. 

    The same bound holds for $\tilde{F}_z$. Indeed, 
$\tilde{x}=\bar x^{-1}x$, where $\bar x$ was chosen so that
\begin{align*}
    |\ell(\bar x)-r_x|\leq c,
    \qquad
    |\ell(\tilde{x})-(\ell(x)-r_x)|\leq c.
\end{align*}
It follows that $\tilde{x}^{-1}$ is a $c$-approximate intermediate
point for $x^{-1}$ at distance $\ell(x)-r_x$ from $e$. Indeed, since $\ell(\tilde{x}^{-1}) = \ell(\tilde{x})$ we have
\begin{align*}
    |\ell(\tilde{x}^{-1})-(\ell(x)-r_x)|\leq c,
\end{align*}
while
$(\tilde{x}^{-1})^{-1}x^{-1}
    =\tilde{x} x^{-1}
    =\bar x^{-1}$,
and thus
\begin{align*}
     \bigl|\ell((\tilde{x}^{-1})^{-1}x^{-1})-r_x\bigr|
    =|\ell(\bar x)-r_x|
    \leq c.
\end{align*} 
These are precisely the two inequalities
for an approximate intermediate point for $x^{-1}$ corresponding to
the 
parameter $\ell(x)-r_x$.

Now, if $v\in\tilde{F}_z$, then for some $x\in X_{m,n,k}$, we have $z=v\tilde{x}$ with $\ell(v)\leq R_x$, 
and therefore
\begin{align*}
    z^{-1}=\tilde{x}^{-1}v^{-1}.
\end{align*}
Noting that under $x \mapsto x^{-1}$ and the interchange of $k$ and $n$, $a_x$ remains unchanged and the corresponding $r$-parameter is $\ell(x) - r_x$, the preceding argument for $\overline{F}_y$ applied mutatis mutandis to $z^{-1}$  bounds the number of possible
values of $v^{-1}$ by $|B_\ell\!\left(2c+\frac72+2\delta\right)|$. 
Since inversion is a bijection, we conclude that
\begin{align*}
    |\tilde{F}_z|
    \leq
    \left|B_\ell\!\left(2c+\frac72+2\delta\right)\right|.
\end{align*}
    Set 
    \begin{align*}
        C := \vert B_\ell(2c + \frac{7}{2} + 2\delta) \vert.
    \end{align*}
    
    By \eqref{eq:multiplicity-convolution-bound-Nyz} above we immediately see    
    \begin{align*}
         \Vert P_m (f*\xi) \Vert^2 \leq  \sum_{y \in E_k} \sum_{z\in E_n} N(y,z) \vert f(y) \vert^2 \cdot \vert \xi(z) \vert^2 \leq C^2 \bigg( \sum_{y\in E_k} \vert f(y) \vert^2 \bigg)\bigg( \sum_{z \in E_n} \vert \xi(z) \vert^2 \bigg),
    \end{align*}
    from which we deduce $\Vert P_m (f*\xi)\Vert \leq C \Vert f\Vert_2 \Vert \xi \Vert_2$ and thus $\Vert P_m \lambda(f) P_n \Vert \leq C \Vert f \Vert_2$. This finishes the proof. 
\end{proof}

\subsection{Controlled off-diagonal cutoff}

Let $\{\delta_x\}_{x\in G}$ be the canonical basis for $\ell^2(G)$. For $T \in B(\ell^2(G))$, define
\begin{align*}
    T_{x,y} := \langle T \delta_y, \delta_x \rangle. 
\end{align*}
We may identify $T$ with the $G\times G$-indexed infinite matrix $(T_{x,y})_{x,y}$. 

Let $\alpha_t(T) = e^{2\pi itD}Te^{-2\pi itD}$ for $T \in B(\ell^2(G))$. Note that $t \mapsto \alpha_t(T)$ may not be norm continuous, but it is always strong operator topology continuous. 
We record the following useful identities. 

\begin{lemma}\label{lemma:matrix-coeff-commutator-identities}
    For any $T \in  B(\ell^2(G))$ we have
    \begin{align*}
         \alpha_t(T)_{x,y} = e^{2\pi it(\ell(x) - \ell(y))} T_{x,y}.
    \end{align*}
    Furthermore, suppose $T\operatorname{Dom}(D)\subseteq\operatorname{Dom}(D)$ and the commutator $[D,T]$, initially defined on $\operatorname{Dom}(D)$, extends to an element of $B(\ell^2(G))$. Then
    \begin{align*}
        [D,T]_{x,y} = (\ell(x) - \ell(y)) T_{x,y},
    \end{align*}
    where we have also denoted the extension by $[D,T]$. This is in particular true for $T = \lambda(f)$ for $f \in C_c(G)$. 
\end{lemma}

\begin{proof}
    Noting that 
    $ e^{\pm 2\pi itD}\delta_x = e^{\pm 2\pi it\ell(x)} \delta_x$ 
    we may calculate
    \begin{align*}
        \alpha_t(T)_{x,y} &= \langle e^{2\pi itD}Te^{-2\pi itD} \delta_y, \delta_x \rangle = \langle T e^{-2\pi itD} \delta_y, e^{-2\pi itD}\delta_x \rangle \\
        &= e^{2\pi it(\ell(x) - \ell(y))} \langle T\delta_y, \delta_x \rangle =  e^{2\pi it(\ell(x) - \ell(y))} T_{x,y}.
    \end{align*}
    The second statement follows by the calculation
    \begin{align*}
        [D,T]_{x,y} :&= \langle [D,T]\delta_y, \delta_x \rangle = \langle T \delta_y , D\delta_x \rangle - \langle TD\delta_y , \delta_x\rangle \\
        &= \ell(x) \langle T \delta_y , \delta_x \rangle - \ell(y) \langle T \delta_y , \delta_x \rangle = (\ell(x) - \ell(y)) T_{x,y}. 
    \end{align*}
    
\end{proof}

For the next lemma we need some notation. 
For $h \colon \R \to \C$, and any $R \geq 1$, denote by $h_R$ the function
\begin{align}\label{eq:R-scaling}
    h_R(t) = h(\frac{t}{R}).
\end{align}
If its (inverse) Fourier transform exists, we denote it by $\check{h}$ with the convention
\begin{align}\label{eq:inv-FT-convention}
    \check{h}(t) = \int_\R h(u) e^{2\pi itu} du.
\end{align}

The following lemma generalizes the main result of \cite[Section 2]{OzawaRieffel2005}. The fact that we are working with length functions which are not integer-valued complicates matters.

\begin{lemma}\label{lemma:smooth-offdiagonal-cutoff}
    Let $h \in C^\infty_c(\R)$ be an even function which takes values in $[0,1]$ and is such that 
    \begin{align*}
        h(t) = \begin{cases}
            1 & \text{if $|t|\leq 1$} \\
            0 & \text{if $|t| \geq 2$} .
        \end{cases}
    \end{align*}
    Let $R \geq 1$, and define
    \begin{align*}
        g_{R}(t) := 
        \begin{cases}
            \frac{1-h_R(t)}{t} & \text{if $t\neq 0$} \\
            0 & \text{if $t = 0$},
        \end{cases}
    \end{align*}
    where $h_R$ is defined as in \eqref{eq:R-scaling}. 
    Set $Q_{h,R} \colon B(\ell^2(G)) \to B(\ell^2(G))$ to be the linear map
    \begin{align*}
        Q_{h,R} T := \int_\R  \check{h_R}(t) \alpha_t(T) dt,
    \end{align*}
    where $\check{h_R}$ is the inverse Fourier transform of $h_R$. We interpret the integral in the strong operator topology sense. 
    Then 
    \begin{align*}
        P_m ( Q_{h,R}T) P_n = 0
    \end{align*}
    for $|m-n| > 2R+2$, and if $T \operatorname{Dom}(D) \subseteq \operatorname{Dom}(D)$ and $[D,T]$ extends boundedly to an operator on $\ell^2(G)$, we have
    \begin{align*}
        \Vert T - Q_{h,R}T \Vert \leq \frac{\Vert \check{g_1}\Vert_1}{R} \Vert [D,T] \Vert,
    \end{align*}
    where $\check{g_1}$ is the inverse Fourier transform of $g_1$.
\end{lemma}

\begin{proof}
    Note first that $h_R \in C^\infty_c(\R)$ for all $R \geq 1$, and so $\check{h_R}$ makes sense and $\Vert \check{h_R} \Vert_1 < \infty$. Likewise, due to the assumptions on $h$, we realize $g_1, g_1' \in L^2(\R)$, and  $\check{g_1} \in L^1(\R)$, i.e. $\Vert \check{g_1}\Vert_1 < \infty$. Now, using the convention of \eqref{eq:inv-FT-convention} and noting that $g_1(t/R) = R g_R (t)$, straightforward calculations will show that
    \begin{align*}
        \Vert \check{h_R} \Vert_1 = \Vert \check{h_1}\Vert_1 \quad \text{and} \quad \Vert \check{g_R} \Vert_1 = \frac{1}{R} \Vert \check{g_1} \Vert_1.
    \end{align*}
    In particular, for $T \in B(\ell^2(G))$ we have
    \begin{align*}
        \Vert Q_{h,R}T \Vert = \Vert \int_\R  \check{h_R}(t) \alpha_t(T) dt \Vert \leq \int_\R \vert \check{h_R}\vert \Vert \alpha_t(T) \Vert dt\leq \Vert \check{h_R}\Vert_1 \cdot \Vert T \Vert,
    \end{align*}
    where we have used $\Vert \alpha_t(T) \Vert = \Vert T \Vert$. Thus $Q_{h,R} \colon B(\ell^2(G)) \to B(\ell^2(G)) $ is a bounded linear map. 

    Using Lemma \ref{lemma:matrix-coeff-commutator-identities}, we observe that for $x,y \in G$ we have
    \begin{align*}
        (Q_{h,R}T)_{x,y} &= \langle Q_{h,R}T \delta_y , \delta_x \rangle = \langle \int_\R \check{h_R}(t) \alpha_t(T) dt \delta_y, \delta_x \rangle \\
        &= \int_\R  \check{h_R}(t)  \langle \alpha_t (T) \delta_y, \delta_x \rangle dt 
        = \int_\R \check{h_R}(t) \alpha_t(T)_{x,y} dt \\
        &= \int \check{h_R}(t) e^{2\pi it(\ell(x) - \ell(y))} T_{x,y} dt 
        = \int \check{h_R}(t) e^{2\pi it(\ell(x) - \ell(y))}  dt T_{x,y} \\
        &= h_R(\ell(y) - \ell(x)) T_{x,y} 
        = h_R(\ell(x) - \ell(y)) T_{x,y},
    \end{align*}
    where in the last step we used that $h_R$ is an even function. Then we immediately see that
    \begin{align*}
        P_m Q_{h,R}(T) P_n = 0 \quad\text{if $|m-n| > 2R +2$}. 
    \end{align*}
    Using Lemma \ref{lemma:matrix-coeff-commutator-identities} again, we may also calculate
    \begin{align*}
        (\int_\R \check{g_R}(t) \alpha_t([D,T]) dt)_{x,y} &= \langle \int_\R \check{g_R}(t) \alpha_t([D,T]) dt \delta_y, \delta_x \rangle =  \int_\R \check{g_R}(t) \langle \alpha_t ([D,T]) \delta_y, \delta_x \rangle dt \\
        &= \int_\R \check{g_R}(t) (\alpha_t([D,T]))_{x,y} dt = \int_\R \check{g_R}(t) e^{2\pi it(\ell(x) - \ell(y))} [D,T]_{x,y} dt \\
        &= \int_\R \check{g_R}(t) e^{2\pi it(\ell(x) - \ell(y))} (\ell(x) - \ell(y)) T_{x,y} dt \\
        &= (\ell(x) - \ell(y)) \cdot g_R(\ell(y) - \ell(x)) T_{x,y} \\
        &= -(1- h_R(\ell(y)-\ell(x))) T_{x,y}\\ &= -(1 - h_R(\ell(x) - \ell(y))) T_{x,y},
    \end{align*}
    since $h_R$ is assumed to be even. 
    Comparing with the previous calculation we deduce that
    \begin{align*}
        T - Q_{h,R}T = - \int_\R \check{g_R}(t) \alpha_t([D,T]) dt.
    \end{align*}
    The integral converges in the strong operator topology as $ [D,T] \in B(\ell^2(G))$ and $\Vert \check{g_R} \Vert_1 \leq R^{-1} \Vert \check{g_1} \Vert_1 < \infty$. It follows that
    \begin{align*}
        \Vert T - Q_{h,R}T \Vert \leq R^{-1} \Vert \check{g_1} \Vert_1 \cdot \Vert [D,T] \Vert. 
    \end{align*}
\end{proof}

\subsection{Proof of the main theorem}

\begin{theorem}\label{thm:main-theorem}
    Let $G$ be a  hyperbolic group, and suppose $\ell\colon G \to [0,\infty)$ is a proper length function for which $(G,d_\ell)$ is a hyperbolic and weakly geodesic metric space. Then $(C_c(G), L)$ is a compact quantum metric space. 
\end{theorem}

\begin{proof}
   Having established Lemma \ref{lemma:annular-norm-control} and Lemma \ref{lemma:smooth-offdiagonal-cutoff}, we are in a position to verify the condition from Proposition \ref{prop:CQMS-equiv-weak-amenable-gps} and thus show that $(C_c(G), L)$ is a compact quantum metric space. 

    As noted previously, $G$ being hyperbolic implies it is weakly amenable. Let $(\phi_n)_n \subseteq C_c(G)$ be a sequence witnessing the weak amenability. 
   Set $M := \sup_n \Vert M_{\phi_n} \Vert_{\rm cb}$, which is finite by assumption. Moreover, $\phi_n(x) \to 1$ for every $x \in G$. 

   By \cite[Proposition D.6]{BrownOzawa08}, $M_{\phi_n}$ extends to a bounded map $\tilde{M}_{\phi_n}$ on $B(\ell^2(G))$, with
\begin{align*}
    \tilde{M}_{\phi_n}(T)_{x,y} = \phi_n(xy^{-1}) T_{x,y}
\end{align*}
   and
   \begin{align*}
       \Vert \tilde{M}_{\phi_n} \Vert_{\rm cb} = \Vert M_{\phi_n} \Vert_{\rm cb} \leq M \quad \text{for all $n$}.
   \end{align*}
   As observed in \cite[Lemma 3.12]{AustadGpdCQMS} 
   \begin{align*}
       [D, \lambda(M_{\phi_n}(f))] = \tilde{M}_{\phi_n} ([D, \lambda(f)]) 
   \end{align*}
    for all $f \in C_c(G)$. 
    To ease notation in the sequel, we write
    \begin{align*}
        f^{(n)} := f - M_{\phi_n}(f).
    \end{align*}
    It then follows that
    \begin{align}\label{eq:f-phinf-commutator-bound}
        \Vert [D, \lambda(f^{(n)})] \Vert \leq (1 + M) \Vert [D, \lambda(f)] \Vert
    \end{align}
    Moreover, there is  an upper bound on the modulus of $\phi_n(x)$ as
    \begin{align*}
        \vert \phi_n (x) \vert = \Vert M_{\phi_n}(\lambda_x) \Vert \leq \Vert M_{\phi_n} \Vert_{\rm cb} = M, 
    \end{align*}
    uniformly in $n$. As such we deduce that
    \begin{align*}
        | 1- \phi_n(x) | \leq (1+M) ,
    \end{align*}
    which we will use below.

   For $f \in C_c(G)$ we write $f = \sum_{k \in \N_0} f_k$, where $f_k = f_{\vert_{E_k}}$ is the restriction of $f$ to the $k$-annulus $E_k = \{x \in G \mid k \leq \ell(x) < k+1 \}$. By the trivial estimate
   \begin{align*}
       \sum_{x \in G} |f(x)|^2 \ell(x)^2 = \Vert [D,\lambda(f) ] \delta_e \Vert_2^2 \leq \Vert [D,\lambda(f)] \Vert ^2
   \end{align*}
   and using $\ell(x) \geq k$ for $x \in E_k$, we obtain
   \begin{align}\label{eq:lower-annulus-L-bound}
       \sum_{k \in \N} k^2 \Vert f_k\Vert_2^2 \leq \Vert [D,\lambda(f)]\Vert^2. 
   \end{align}
   Then, by Lemma \ref{lemma:smooth-offdiagonal-cutoff} and \eqref{eq:f-phinf-commutator-bound}
   \begin{align*}
       \Vert \lambda(f^{(n)})-Q_{h,R}(\lambda(f^{(n)}))\Vert &\leq \frac{\Vert\check{g_1}\Vert_1}{R} \Vert [D, \lambda(f^{(n)}) ]\Vert \notag \\
       &\leq \frac{\Vert\check{g_1}\Vert_1}{R} \cdot (1+M)\cdot  \Vert [D, \lambda(f)] \Vert,
   \end{align*}
   where $g_1$ is as in Lemma \ref{lemma:smooth-offdiagonal-cutoff}. 

   As $P_m, P_n$ commute with $D$, we may write
   \begin{align*}
       P_m Q_{h,R}(\lambda(f_k))P_n = \int_\R \check{h_R}(t) \alpha_t (P_m \lambda(f_k) P_n) dt,
   \end{align*}
   and so Lemma \ref{lemma:annular-norm-control} and the fact that $\Vert \check{h_R} \Vert_1 = \Vert \check{h_1}\Vert_1$ yields
   \begin{align*}
       \Vert P_m Q_{h,R}(\lambda(f_k))P_n \Vert &\leq \int_\R \vert \check{h_R}(t) \vert \Vert \alpha_t(P_m \lambda(f_k) P_n)\Vert dt \\
       &= \Vert \check{h_R} \Vert_1 \Vert P_m \lambda(f_k) P_n \Vert \\
       &\leq \Vert \check{h_1} \Vert_1 C \Vert f_k\Vert_2.
   \end{align*}
    Noting that $P_mQ_{h,R}(T)P_n=0$ for $|m-n| > 2R+2$, we write 
    \begin{align*}
        Q_{h,R}(\lambda(f)) = \sum_{|j|\leq \lceil 2R +2 \rceil} \sum_{m \in \N_0} P_m Q_{h,R}(\lambda(f)) P_{m-j},
    \end{align*}
    where we interpret $P_{m-j} = 0$ if $m-j < 0$. 
    
    For each fixed $j$, the sum over $m$ is understood in the strong operator topology. Since the projections $P_m$ are mutually orthogonal, both the domain projections $P_{m-j}$ and the range projections $P_m$ in this sum are pairwise orthogonal. Analogously to \cite{OzawaRieffel2005}, we deduce
    \begin{align*}
        \Vert \sum_{m\in \N_0} P_m Q_{h,R}(\lambda(f_k)) P_{m-j}\Vert = \sup_{m \in \N_0} \Vert P_m Q_{h,R}(\lambda(f_k)) P_{m-j} \Vert .
    \end{align*}
    Then
    \begin{align*}
        \Vert Q_{h,R}(\lambda(f)) \Vert &= \Vert \sum_{k \in \N_0}\sum_{|j|\leq \lceil 2R+2\rceil} \sum_{m\in \N_0} P_m Q_{h,R} (\lambda(f_k)) P_{m-j} \Vert \notag \\
        &\leq \sum_{k \in \N_0} \sum_{|j|\leq \lceil 2R +2 \rceil} \Vert \sum_{m\in \N_0} P_m Q_{h,R}(\lambda(f_k)) P_{m-j}\Vert \notag \\
        &\leq \sum_{k \in \N_0} \sum_{|j|\leq \lceil 2R +2 \rceil} \sup_{m \in \N_0} \Vert P_m Q_{h,R}(\lambda(f_k)) P_{m-j} \Vert  \notag \\
        &\leq (2 \lceil 2R +2 \rceil +1) \Vert \check{h_1}\Vert_1 C \sum_{k\in \N_0} \Vert f_k \Vert_2
    \end{align*}

   We proceed to bound the norm of the remainder for large values of $k$.
   Note that by \eqref{eq:lower-annulus-L-bound} and using Cauchy-Schwarz we have, for any $K \geq 1$
   \begin{align*}
       \sum_{k > K} \Vert f_k-M_{\phi_n}(f_k) \Vert_2 &\leq (1+M) \sum_{k > K } \Vert f_k \Vert_2 = (1+M) \sum_{k > K} \frac{1}{k} k\Vert f_k \Vert_2 \notag\\
       &\leq (1+M) \bigg( \sum_{k > K} \frac{1}{k^2} \bigg)^{1/2} \cdot \bigg( \sum_{k > K} k^2 \Vert f_k \Vert_2^2 \bigg)^{1/2}\notag  \\
       &\leq (1+M) \bigg( \sum_{k>K} \frac{1}{k^2} \bigg)^{1/2} \cdot L(f).
   \end{align*}
    The necessary estimate for small values of $k$ will come from pointwise convergence of the sequence $(\phi_n)_n$ to $1$. Note that the ball $B_\ell(K+1)$ is finite as $\ell$ is proper. Since $(\phi_n)_n \to 1$ pointwise, we can then for any $\eta >0$ choose $N$ so that for all $n \geq N$ 
    the value $\max_{x\in B_\ell(K+1)}\vert 1 -\phi_n(x) \vert < \eta$. 
    If $G = \{e\}$ the result is immediate, and so we proceed to assume $G \neq \{e\}$. 
    Then, since $\ell$ is proper, $B_\ell(1)$ is finite, and so there is $d_0 \geq 1$ such that
    \begin{align*}
        d_0 \geq (\inf_{x\neq e} \ell(x))^{-1}.
    \end{align*}
   Since $f(e) = 0$, we have $\Vert f \Vert_2 \leq d_0 \cdot L(f)$. Thus
   \begin{align*}
       \sum_{k=0}^K \Vert f_k - M_{\phi_n}(f_k) \Vert_2 &\leq (\max_{x \in B_\ell(K+1)}\vert 1-\phi_n(x) \vert) \sum_{k=0}^K \Vert f_k \Vert_2 \notag \\
       &\leq  (\max_{x \in B_\ell(K+1)}\vert 1-\phi_n(x) \vert) \sqrt{K+1} \bigg( \sum_{k=0}^K \Vert f_k \Vert_2^2 \bigg)^{1/2}      \notag    \\
       &\leq  (\max_{x \in B_\ell(K+1)}\vert 1-\phi_n(x) \vert) \sqrt{K+1} \Vert f \Vert_2 \notag \\
       &\leq  (\max_{x \in B_\ell(K+1)}\vert 1-\phi_n(x) \vert) \cdot d_0 \sqrt{K+1}\cdot L(f) .
   \end{align*}

    Finally, we verify the condition from Proposition \ref{prop:CQMS-equiv-weak-amenable-gps} and show that $(C_c(G), L)$ is a compact quantum metric space. 
   Let $\varepsilon > 0$.  
   Defining
   \begin{align*}
       C_1 = \Vert\check{g_1}\Vert_1 \cdot (1+M) \quad \text{and} \quad C_2 = (2\lceil 2R+2\rceil +1) \cdot \Vert \check{h_1}\Vert_1 \cdot C ,
   \end{align*}
   we calculate
   \begin{align*}
       &\Vert \lambda(f^{(n)}) \Vert \leq \Vert \lambda(f^{(n)}) - Q_{h,R}(\lambda(f^{(n)})) \Vert  + \Vert Q_{h,R} (\lambda(f^{(n)})) \Vert \\
       &\leq \frac{C_1}{R}  L(f) +  C_2  \sum_{k \in \N_0} \Vert (f^{(n)})_k \Vert_2        \\
       &\leq  \frac{C_1}{R}  L(f) + C_2   \big( \sum_{k\leq K} \Vert (f^{(n)})_k \Vert_2 + \sum_{k> K} \Vert (f^{(n)})_k \Vert_2 \big)  \\
       &\leq \frac{C_1}{R}  L(f) + C_2 (\max_{x \in B_\ell(K+1)}\vert 1-\phi_n(x) \vert  d_0 \sqrt{K+1}  L(f) + (1+M) \bigg( \sum_{k>K} \frac{1}{k^2} \bigg)^{1/2} ) L(f) \\
       &= \bigg( \frac{C_1}{R} + C_2 (1+M) \bigg( \sum_{k>K} \frac{1}{k^2} \bigg)^{1/2} + C_2 d_0\sqrt{K+1}\max_{x \in B_\ell(K+1)}\vert 1-\phi_n(x) \vert \bigg) L(f).
   \end{align*}
   Remembering that $L(f) \leq 1$, we now first choose $R$ such that $C_1 R^{-1} < \varepsilon/3$. Then $C_2$ is some fixed finite value, and we choose $K$ so that 
   \begin{align*}
       \bigg(\sum_{k>K} \frac{1}{k^2} \bigg)^{1/2} < \frac{\varepsilon}{3 \cdot C_2 \cdot (1+M)}.
   \end{align*}
   Lastly we choose $N$ so that for all $n \geq N$, 
   \begin{align*}
       \max_{x \in B_\ell(K+1)} \vert 1- \phi_n(x) \vert < \frac{\varepsilon}{3 C_2 \cdot d_0\sqrt{K+1}}.
   \end{align*}
   Thus $\Vert \lambda(f^{(n)}) \Vert < \varepsilon$. Since $f \in E$ was arbitrary, the result follows. 
\end{proof}

\begin{remark}
    The length functions appearing in Theorem \ref{thm:main-theorem} also have the rapid decay property. Indeed, by Lemma \ref{lemma:proper-weak-geodesic-is-fin-gen}, if $\ell$ is proper and weakly geodesic, then there is a finite symmetric generating set $S$ such that $\ell$ and the corresponding word length $\ell_S$ are quasi-isometric. It is known that word hyperbolic groups have rapid decay with respect to word length functions \cite{delaHarpeHyperbolic88, JolissaintRD90}. In other words, there are constants $C,s \geq 0$ such that
    \begin{align*}
        \Vert \lambda(f) \Vert \leq C \cdot \bigg( \sum_{g \in G} \vert f(g) \vert^2 (1+ \ell_S(g) )^{2s} \bigg)^{1/2}
    \end{align*}
    for all $f \in C_c(G)$. From this expression it is easy to see that rapid decay is preserved through quasi-isometry of length functions, and so $\ell$ also has rapid decay.

    However, Theorem \ref{thm:main-theorem} does not follow from this observation, because it is not known whether the compact quantum metric property for the first commutator seminorm \eqref{eq:def-L-ell} is invariant under quasi-isometric changes of length function. 
\end{remark}

\subsection{Examples}\label{subsec:Examples}
In this section we show how to construct compact quantum metric spaces from some explicit examples of length functions on hyperbolic groups. The examples below are far from an exhaustive list of candidates. Any left-invariant metric $d$ on a hyperbolic group $G$ gives rise to a length function. In light of Lemma \ref{lemma:rough-geodesic-lfs-on-hyperbolic-groups}, we deduce that if this metric is hyperbolic and quasi-isometric to a word length function, then weak geodesicity of $d$ and properness of $\ell$ is automatic, and so we fall under the purview of Theorem \ref{thm:main-theorem} and can construct compact quantum metric spaces. We cover two families of examples in particular: length functions arising from random walks on groups, and length functions arising from proper cocompact isometric actions on hyperbolic spaces.  
\subsubsection{Random walks on hyperbolic groups}
We refer the reader to \cite{WoessRandomWalks2000} for standard facts about random walks on groups. 
Let $G$ be a non-elementary hyperbolic group, and suppose $\mu$ is a symmetric and finitely supported probability measure on $G$ with support $\mathrm{supp}(\mu)$, which we further assume generates $G$. The random walk on $G$ with law $\mu$ has state space $G$ and transition probabilities $p_\mu (x,y) = \mu(x^{-1}y)$. Then the $n$-step transition probabilities are $p^{(n)}_\mu (x,y ) = \mu^{*n}(x^{-1}y)$, where $\mu^{*n}$ is the $n$-fold convolution of $\mu$ with itself, and convolution of measures $\mu$ and $\nu$ is given by
\begin{align*}
    \mu * \nu (x) = \sum_{y \in G} \mu(y) \nu(y^{-1}x) =\sum_{y\in G} \mu(xy^{-1}) \nu(y).
\end{align*}
Denote by $F_\mu(x,y)$ the hitting probabilities of the random walk. 
Then 
\begin{align*}
    d_\mu(x,y) := -\log F_\mu(x,y) 
\end{align*}
is a left-invariant and hyperbolic metric which is quasi-isometric to the word metric 
on $G$ \cite[Corollary 1.2]{BlachereHaissinskyMathieu2011}. The metric $d_\mu$  is known as the Green metric associated with $\mu$. We thus have a length function
\begin{align*}
    \ell_\mu (x) = d_\mu(e,x) = -\log F_\mu (e,x).
\end{align*}
 Since $d_{\ell_\mu} = d_\mu$, it now follows from Lemma \ref{lemma:rough-geodesic-lfs-on-hyperbolic-groups} that $\ell_\mu$ induces a weakly geodesic metric on $G$. Moreover, quasi-isometry with the word length function implies $\ell_\mu$ is proper. Thus the following is immediate from Theorem \ref{thm:main-theorem}.

\begin{corollary}\label{cor:CQMS-from-random-walk}
Let $G$ be a non-elementary hyperbolic group, and 
let $\mu$ be a 
finitely supported symmetric probability measure for which $\mathrm{supp}(\mu)$ generates $G$. Denote by $\ell_\mu$ the proper length function associated with the Green metric of $\mu$. Let $L$ be the associated seminorm as in \eqref{eq:def-L-ell}. Then $(C_c(G), L)$ is a compact quantum metric space.
\end{corollary}

\subsubsection{Proper cocompact isometric actions on hyperbolic spaces}
Suppose $(X,d)$ is a proper, geodesic and hyperbolic metric space, and let $G \acts X$ be an  
isometric left action. 
Suppose there is a point $x_0 \in X$ with ${\rm Stab}_G (x_0) = \{e\}$, 
and define a map $\ell \colon G \to [0,\infty)$ through
\begin{align}\label{eq:length-function-from-action}
    \ell(g) = d(x_0 , g x_0)
\end{align}
for $g \in G$. Using the fact that the action is isometric together with ${\rm Stab}_G(x_0) = \{e\}$, it is straightforward to verify that $\ell$ is a length function on $G$. In fact, we then have
\begin{align*}
    d_{\ell}(g,h) = \ell(g^{-1}h) = d(x_0 , g^{-1}hx_0) = d(gx_0, hx_0)
\end{align*}
which implies that $g \mapsto g x_0$ is an isometric embedding of $G$ into $X$ whose image is the orbit $Gx_0$. Since $X$ is assumed hyperbolic, it follows that $Gx_0$ is hyperbolic, and thus $(G, d_\ell)$ is also hyperbolic. 

If we furthermore assume that
\begin{align}\label{eq:action-lf-proper}
    |\{g \in G \mid d(x_0, gx_0) \leq r \} |< \infty \quad \text{ for all $r \geq 0$},
\end{align}
we immediately get that 
$\ell$ is proper. The property \eqref{eq:action-lf-proper} is also implied by $X$ being proper if we additionally assume that the action is a proper action, that is, for each compact $K \subseteq X$ the set $\{g \in G \mid gK \cap K \neq \emptyset \}$ is finite, see \cite[I.8.3]{BridsonHaefligerMetricSpacesBook}. To see this, let $r\geq 0$ be arbitrary and let $K = B(x_0, r)$. Then $K$ is compact by properness of $X$, and 
\begin{align*}
    \{g \in G \mid \ell(g) \leq r \} = \{g \in G \mid d(x_0 , gx_0) \leq r \} \subseteq \{g \in G \mid gK \cap K \neq \emptyset \},
\end{align*}
showing that $\ell$ is a proper length function. 

If we additionally assume that the action is cocompact, it follows from the \v{S}varc--Milnor lemma \cite[Proposition I.8.19]{BridsonHaefligerMetricSpacesBook} that $\ell$ is quasi-isometric to a word length function on $G$. By Lemma \ref{lemma:rough-geodesic-lfs-on-hyperbolic-groups} we deduce that $\ell$ is weakly $c$-geodesic for some $c\geq 0$. We have the following result.

\begin{corollary}
    Suppose $G$ is word hyperbolic, and let $G \acts X$ be a proper cocompact and isometric action on a proper geodesic hyperbolic metric space $(X,d)$. Suppose $x_0 \in X$ satisfies ${\rm Stab}_G(x_0) = \{e\}$, and let $\ell$ be the length function from 
    \eqref{eq:length-function-from-action}. With $L \colon C_c(G) \to [0,\infty)$ as in \eqref{eq:def-L-ell}, the pair $(C_c(G), L)$ is a compact quantum metric space.
\end{corollary}

 \vspace{1cm}

 \noindent {\bf Declaration of AI use: } The writing of this paper was assisted by ChatGPT-5.6 Sol, a large language model made by OpenAI. It  suggested a number of changes to the exposition, the primary being a restructuring of the original argument for Lemma \ref{lemma:annular-norm-control}. All AI-produced suggestions have been verified, and where necessary corrected, by the author. 

\printbibliography

@article {AguilarKaad2018,
    AUTHOR = {Aguilar, Konrad and Kaad, Jens},
     TITLE = {The {P}odle\'s{} sphere as a spectral metric space},
   JOURNAL = {J. Geom. Phys.},
  FJOURNAL = {Journal of Geometry and Physics},
    VOLUME = {133},
      YEAR = {2018},
     PAGES = {260--278},
      ISSN = {0393-0440,1879-1662},
   MRCLASS = {58B34 (46L30 46L89 58B32)},
  MRNUMBER = {3850270},
MRREVIEWER = {Ferdinand\ Ngakeu},
}

@misc{AustadGpdCQMS,
Author = {Are Austad},
Title = {Quantum metrics from length functions on étale groupoids},
Year = {2026},
Eprint = {arXiv:2602.20032},
}

@article {AustadKaadKyed2025,
    AUTHOR = {Austad, Are and Kaad, Jens and Kyed, David},
     TITLE = {Quantum metrics on crossed products with groups of polynomial
              growth},
   JOURNAL = {Trans. Amer. Math. Soc.},
  FJOURNAL = {Transactions of the American Mathematical Society},
    VOLUME = {378},
      YEAR = {2025},
    NUMBER = {3},
     PAGES = {1939--1973},
      ISSN = {0002-9947,1088-6850},
   MRCLASS = {58B34 (46L55 47L65)},
  MRNUMBER = {4866354},
}

@article {AustadKyed2026,
    AUTHOR = {Austad, Are and Kyed, David},
     TITLE = {Quantum metrics from length functions on quantum groups},
   JOURNAL = {J. Funct. Anal.},
  FJOURNAL = {Journal of Functional Analysis},
    VOLUME = {290},
      YEAR = {2026},
    NUMBER = {4},
     PAGES = {Paper No. 111256, 32},
      ISSN = {0022-1236,1096-0783},
   MRCLASS = {58B34 (46L89 58B32)},
  MRNUMBER = {4983378},
}

@misc{BellissardMarcolliReihani2010,
Author = {Jean V. Bellissard and Matilde Marcolli and Kamran Reihani},
Title = {Dynamical Systems on Spectral Metric Spaces},
Year = {2010},
Eprint = {arXiv:1008.4617},
}

@article {BhowmickVoigtZacharias2015,
    AUTHOR = {Bhowmick, Jyotishman and Voigt, Christian and Zacharias,
              Joachim},
     TITLE = {Compact quantum metric spaces from quantum groups of rapid
              decay},
   JOURNAL = {J. Noncommut. Geom.},
  FJOURNAL = {Journal of Noncommutative Geometry},
    VOLUME = {9},
      YEAR = {2015},
    NUMBER = {4},
     PAGES = {1175--1200},
      ISSN = {1661-6952,1661-6960},
   MRCLASS = {81R05 (58B34 81R50)},
  MRNUMBER = {3448333},
MRREVIEWER = {Rita\ Fioresi},
}

@article {BlachereHaissinskyMathieu2011,
    AUTHOR = {Blach\`ere, S\'ebastien and Ha\"issinsky, Peter and Mathieu,
              Pierre},
     TITLE = {Harmonic measures versus quasiconformal measures for
              hyperbolic groups},
   JOURNAL = {Ann. Sci. \'Ec. Norm. Sup\'er. (4)},
  FJOURNAL = {Annales Scientifiques de l'\'Ecole Normale Sup\'erieure.
              Quatri\`eme S\'erie},
    VOLUME = {44},
      YEAR = {2011},
    NUMBER = {4},
     PAGES = {683--721},
      ISSN = {0012-9593,1873-2151},
   MRCLASS = {20F67 (31C12 31C45 60B15)},
  MRNUMBER = {2919980},
MRREVIEWER = {Goulnara\ N.\ Arzhantseva},
       DOI = {10.24033/asens.2153},
}

@article {BonkSchramm2000,
    AUTHOR = {Bonk, M. and Schramm, O.},
     TITLE = {Embeddings of {G}romov hyperbolic spaces},
   JOURNAL = {Geom. Funct. Anal.},
  FJOURNAL = {Geometric and Functional Analysis},
    VOLUME = {10},
      YEAR = {2000},
    NUMBER = {2},
     PAGES = {266--306},
      ISSN = {1016-443X,1420-8970},
   MRCLASS = {53C23 (54E40 57M07)},
  MRNUMBER = {1771428},
MRREVIEWER = {Michel\ Coornaert},
       DOI = {10.1007/s000390050009},
}

@book {BridsonHaefligerMetricSpacesBook,
    AUTHOR = {Bridson, Martin R. and Haefliger, Andr\'e},
     TITLE = {Metric spaces of non-positive curvature},
    SERIES = {Grundlehren der mathematischen Wissenschaften [Fundamental
              Principles of Mathematical Sciences]},
    VOLUME = {319},
 PUBLISHER = {Springer-Verlag, Berlin},
      YEAR = {1999},
     PAGES = {xxii+643},
      ISBN = {3-540-64324-9},
   MRCLASS = {53C23 (20F65 53C70 57M07)},
  MRNUMBER = {1744486},
MRREVIEWER = {Athanase\ Papadopoulos},
       DOI = {10.1007/978-3-662-12494-9},
}

@book {BrownOzawa08,
    AUTHOR = {Brown, Nathanial P. and Ozawa, Narutaka},
     TITLE = {{$C^*$}-algebras and finite-dimensional approximations},
    SERIES = {Graduate Studies in Mathematics},
    VOLUME = {88},
 PUBLISHER = {American Mathematical Society, Providence, RI},
      YEAR = {2008},
     PAGES = {xvi+509},
      ISBN = {978-0-8218-4381-9; 0-8218-4381-8},
   MRCLASS = {46L05 (43A07 46-02 46L10)},
  MRNUMBER = {2391387},
MRREVIEWER = {Mikael\ R\o rdam},
       DOI = {10.1090/gsm/088},
}

@article {CantrellTanaka2025,
    AUTHOR = {Cantrell, Stephen and Tanaka, Ryokichi},
     TITLE = {The {M}anhattan curve, ergodic theory of topological flows and
              rigidity},
   JOURNAL = {Geom. Topol.},
  FJOURNAL = {Geometry \& Topology},
    VOLUME = {29},
      YEAR = {2025},
    NUMBER = {4},
     PAGES = {1851--1907},
      ISSN = {1465-3060,1364-0380},
   MRCLASS = {20F67 (37D35 37D40)},
  MRNUMBER = {4929470},
       DOI = {10.2140/gt.2025.29.1851},
}

@article {ChristRieffel2017,
    AUTHOR = {Christ, Michael and Rieffel, Marc A.},
     TITLE = {Nilpotent group {${\rm C}^*$}-algebras as compact quantum
              metric spaces},
   JOURNAL = {Canad. Math. Bull.},
  FJOURNAL = {Canadian Mathematical Bulletin. Bulletin Canadien de
              Math\'ematiques},
    VOLUME = {60},
      YEAR = {2017},
    NUMBER = {1},
     PAGES = {77--94},
      ISSN = {0008-4395,1496-4287},
   MRCLASS = {46L87 (20F65 22D15 53C23 58B34)},
  MRNUMBER = {3612100},
MRREVIEWER = {Anton\ Yu.\ Savin},
       DOI = {10.4153/CMB-2016-040-6},
}

@article {ChristensenIvanRD,
    AUTHOR = {Antonescu, Cristina and Christensen, Erik},
     TITLE = {Metrics on group {$C^*$}-algebras and a non-commutative
              {A}rzel\`a-{A}scoli theorem},
   JOURNAL = {J. Funct. Anal.},
  FJOURNAL = {Journal of Functional Analysis},
    VOLUME = {214},
      YEAR = {2004},
    NUMBER = {2},
     PAGES = {247--259},
      ISSN = {0022-1236,1096-0783},
   MRCLASS = {46L85 (46L07 58B34)},
  MRNUMBER = {2083302},
MRREVIEWER = {Eric Leichtnam},
}

@article {Connes1989,
    AUTHOR = {Connes, A.},
     TITLE = {Compact metric spaces, {F}redholm modules, and
              hyperfiniteness},
   JOURNAL = {Ergodic Theory Dynam. Systems},
  FJOURNAL = {Ergodic Theory and Dynamical Systems},
    VOLUME = {9},
      YEAR = {1989},
    NUMBER = {2},
     PAGES = {207--220},
      ISSN = {0143-3857,1469-4417},
   MRCLASS = {46L80 (19K56 22D25 58B15 58G12)},
  MRNUMBER = {1007407},
MRREVIEWER = {Jeffrey\ Fox},
}

@book {ConnesNCGBook,
    AUTHOR = {Connes, Alain},
     TITLE = {Noncommutative geometry},
 PUBLISHER = {Academic Press Inc., San Diego, CA},
      YEAR = {1994},
     PAGES = {xiv+661},
      ISBN = {0-12-185860-X},
   MRCLASS = {46Lxx (19K56 22D25 58B30 58G12 81T13 81V22 81V70)},
  MRNUMBER = {1303779},
MRREVIEWER = {John\ Roe},
}

@article {delaHarpeHyperbolic88,
    AUTHOR = {de la Harpe, Pierre},
     TITLE = {Groupes hyperboliques, alg\`ebres d'op\'erateurs et un
              th\'eor\`eme de {J}olissaint},
   JOURNAL = {C. R. Acad. Sci. Paris S\'er. I Math.},
  FJOURNAL = {Comptes Rendus des S\'eances de l'Acad\'emie des Sciences.
              S\'erie I. Math\'ematique},
    VOLUME = {307},
      YEAR = {1988},
    NUMBER = {14},
     PAGES = {771--774},
      ISSN = {0249-6291},
   MRCLASS = {22D25 (20F38 22E40 46L80 58G12)},
  MRNUMBER = {972078},
MRREVIEWER = {Alain\ Valette},
}

@article {FarsiLandryLarsenPacker,
    AUTHOR = {Farsi, Carla and Landry, Therese and Larsen, Nadia S. and
              Packer, Judith},
     TITLE = {Spectral triples for noncommutative solenoids and a {W}iener's
              lemma},
   JOURNAL = {J. Noncommut. Geom.},
  FJOURNAL = {Journal of Noncommutative Geometry},
    VOLUME = {18},
      YEAR = {2024},
    NUMBER = {4},
     PAGES = {1415--1452},
      ISSN = {1661-6952,1661-6960},
   MRCLASS = {47B07 (22D15 46L87 58B34)},
  MRNUMBER = {4793458},
MRREVIEWER = {Botao\ Long},
       DOI = {10.4171/jncg/557},
}

@misc{GerontogiannisMesland25,
Author = {Dimitris Michail Gerontogiannis and Bram Mesland},
Title = {Ideal quantum metrics from fractional Laplacians},
Year = {2025},
Eprint = {arXiv:2502.04187},
}

@article {HawkinsSkalskiWhiteZacharias2013,
    AUTHOR = {Hawkins, A. and Skalski, A. and White, S. and Zacharias, J.},
     TITLE = {On spectral triples on crossed products arising from
              equicontinuous actions},
   JOURNAL = {Math. Scand.},
  FJOURNAL = {Mathematica Scandinavica},
    VOLUME = {113},
      YEAR = {2013},
    NUMBER = {2},
     PAGES = {262--291},
      ISSN = {0025-5521,1903-1807},
   MRCLASS = {46L55 (58B34)},
  MRNUMBER = {3145183},
MRREVIEWER = {Charles\ Batty},
}

@article {JolissaintRD90,
    AUTHOR = {Jolissaint, Paul},
     TITLE = {Rapidly decreasing functions in reduced {$C^*$}-algebras of
              groups},
   JOURNAL = {Trans. Amer. Math. Soc.},
  FJOURNAL = {Transactions of the American Mathematical Society},
    VOLUME = {317},
      YEAR = {1990},
    NUMBER = {1},
     PAGES = {167--196},
      ISSN = {0002-9947,1088-6850},
   MRCLASS = {22D25 (43A15 46L99)},
  MRNUMBER = {943303},
MRREVIEWER = {A.\ Derighetti},
       DOI = {10.2307/2001458},
}

@book {KaadKyedSU2,
    AUTHOR = {Kaad, Jens and Kyed, David},
     TITLE = {The quantum metric structure of quantum {$\rm SU(2)$}},
    SERIES = {Memoirs of the European Mathematical Society},
    VOLUME = {18},
 PUBLISHER = {EMS Press, Berlin},
      YEAR = {2025},
     PAGES = {viii+119},
      ISBN = {978-3-98547-091-4; 978-3-98547-591-9},
   MRCLASS = {58B32 (46Lxx 58B34 81R60)},
  MRNUMBER = {4867562},
}

@article {KasparovSkandalis2003,
    AUTHOR = {Kasparov, Gennadi and Skandalis, Georges},
     TITLE = {Groups acting properly on ``bolic'' spaces and the {N}ovikov
              conjecture},
   JOURNAL = {Ann. of Math. (2)},
  FJOURNAL = {Annals of Mathematics. Second Series},
    VOLUME = {158},
      YEAR = {2003},
    NUMBER = {1},
     PAGES = {165--206},
      ISSN = {0003-486X,1939-8980},
   MRCLASS = {58J22 (19K35 19M05 57R65)},
  MRNUMBER = {1998480},
MRREVIEWER = {Tsuyoshi\ Kato},
       DOI = {10.4007/annals.2003.158.165},
}

@article {KlisseCQMS,
    AUTHOR = {Klisse, Mario},
     TITLE = {Crossed products as compact quantum metric spaces},
   JOURNAL = {Canad. J. Math.},
  FJOURNAL = {Canadian Journal of Mathematics. Journal Canadien de
              Math\'ematiques},
    VOLUME = {78},
      YEAR = {2026},
    NUMBER = {1},
     PAGES = {245--275},
      ISSN = {0008-414X,1496-4279},
   MRCLASS = {46L05 (46L87 46L89 47L65 58B34)},
  MRNUMBER = {5020072},
}

@misc{Klisse26,
Author = {Mario Klisse},
Title = {Word-Length Spectral Triples of $(\mathbb{Z}/2\mathbb{Z})\wr\mathbb{F}_{d}$ Are Not Metric},
Year = {2026},
Eprint = {arXiv:2608.12080},
}

@article {OzawaHyperbolicWeakAmenable2008,
    AUTHOR = {Ozawa, Narutaka},
     TITLE = {Weak amenability of hyperbolic groups},
   JOURNAL = {Groups Geom. Dyn.},
  FJOURNAL = {Groups, Geometry, and Dynamics},
    VOLUME = {2},
      YEAR = {2008},
    NUMBER = {2},
     PAGES = {271--280},
      ISSN = {1661-7207,1661-7215},
   MRCLASS = {20F67 (43A65 46L07)},
  MRNUMBER = {2393183},
MRREVIEWER = {Zoran\ \v Suni\'c},
       DOI = {10.4171/GGD/40},
}

@article {OzawaRieffel2005,
    AUTHOR = {Ozawa, Narutaka and Rieffel, Marc A.},
     TITLE = {Hyperbolic group {$C^*$}-algebras and free-product
              {$C^*$}-algebras as compact quantum metric spaces},
   JOURNAL = {Canad. J. Math.},
  FJOURNAL = {Canadian Journal of Mathematics. Journal Canadien de
              Math\'ematiques},
    VOLUME = {57},
      YEAR = {2005},
    NUMBER = {5},
     PAGES = {1056--1079},
      ISSN = {0008-414X,1496-4279},
   MRCLASS = {46L87 (20F67 46L05 58B34)},
  MRNUMBER = {2164594},
MRREVIEWER = {Paul\ Jolissaint},
       DOI = {10.4153/CJM-2005-040-0},
}

@article {RieffelMetricsOnStateSpaces1999,
    AUTHOR = {Rieffel, Marc A.},
     TITLE = {Metrics on state spaces},
   JOURNAL = {Doc. Math.},
  FJOURNAL = {Documenta Mathematica},
    VOLUME = {4},
      YEAR = {1999},
     PAGES = {559--600},
      ISSN = {1431-0635,1431-0643},
   MRCLASS = {46L87 (58B34)},
  MRNUMBER = {1727499},
MRREVIEWER = {Jacek\ Brodzki},
}

@article {Rieffel02,
    AUTHOR = {Rieffel, Marc A.},
     TITLE = {Group {$C^*$}-algebras as compact quantum metric spaces},
   JOURNAL = {Doc. Math.},
  FJOURNAL = {Documenta Mathematica},
    VOLUME = {7},
      YEAR = {2002},
     PAGES = {605--651},
      ISSN = {1431-0635,1431-0643},
   MRCLASS = {22D25 (20F65 20F69 46L55 46L87 58J42)},
  MRNUMBER = {2015055},
MRREVIEWER = {Alain\ Valette},
}

@incollection {RieffelCQMS04,
    AUTHOR = {Rieffel, Marc A.},
     TITLE = {Compact quantum metric spaces},
 BOOKTITLE = {Operator algebras, quantization, and noncommutative geometry},
    SERIES = {Contemp. Math.},
    VOLUME = {365},
     PAGES = {315--330},
 PUBLISHER = {Amer. Math. Soc., Providence, RI},
      YEAR = {2004},
      ISBN = {0-8218-3402-9},
   MRCLASS = {46L87 (53C23 58B34)},
  MRNUMBER = {2106826},
MRREVIEWER = {Vladimir\ Manuilov},
}

@incollection {Rieffel2004MatrixAlgs,
    AUTHOR = {Rieffel, Marc A.},
     TITLE = {Matrix algebras converge to the sphere for quantum
              {G}romov-{H}ausdorff distance},
      NOTE = {Gromov-Hausdorff distance for quantum metric spaces. Matrix
              algebras converge to the sphere for quantum Gromov-Hausdorff
              distance},
   JOURNAL = {Mem. Amer. Math. Soc.},
  FJOURNAL = {Memoirs of the American Mathematical Society},
    VOLUME = {168},
      YEAR = {2004},
    NUMBER = {796},
     PAGES = {67--91},
      ISSN = {0065-9266,1947-6221},
   MRCLASS = {46L87 (53C23 58B34 81R30)},
  MRNUMBER = {2055928},
}

@incollection {Rieffel2004qGH,
    AUTHOR = {Rieffel, Marc A.},
     TITLE = {Gromov-{H}ausdorff distance for quantum metric spaces},
      NOTE = {Appendix 1 by Hanfeng Li,
              Gromov-Hausdorff distance for quantum metric spaces. Matrix
              algebras converge to the sphere for quantum Gromov-Hausdorff
              distance},
   JOURNAL = {Mem. Amer. Math. Soc.},
  FJOURNAL = {Memoirs of the American Mathematical Society},
    VOLUME = {168},
      YEAR = {2004},
    NUMBER = {796},
     PAGES = {1--65},
      ISSN = {0065-9266,1947-6221},
   MRCLASS = {46L87 (53C23 58B34 60B10)},
  MRNUMBER = {2055927},
}

@incollection {RieffelMartricialBridges2016,
    AUTHOR = {Rieffel, Marc A.},
     TITLE = {Matricial bridges for ``matrix algebras converge to the
              sphere''},
 BOOKTITLE = {Operator algebras and their applications},
    SERIES = {Contemp. Math.},
    VOLUME = {671},
     PAGES = {209--233},
 PUBLISHER = {Amer. Math. Soc., Providence, RI},
      YEAR = {2016},
      ISBN = {978-1-4704-1948-6},
   MRCLASS = {46L87 (53C23 58B34 81R10)},
  MRNUMBER = {3546687},
MRREVIEWER = {Jian\ Wang},
}

@book {WoessRandomWalks2000,
    AUTHOR = {Woess, Wolfgang},
     TITLE = {Random walks on infinite graphs and groups},
    SERIES = {Cambridge Tracts in Mathematics},
    VOLUME = {138},
 PUBLISHER = {Cambridge University Press, Cambridge},
      YEAR = {2000},
     PAGES = {xii+334},
      ISBN = {0-521-55292-3},
   MRCLASS = {60B15 (60G50 60J10)},
  MRNUMBER = {1743100},
MRREVIEWER = {Donald\ I.\ Cartwright},
       DOI = {10.1017/CBO9780511470967},
}

\end{document}